\documentclass[10pt]{article}

\usepackage[nottoc,notlot,notlof]{tocbibind}
\usepackage{amsmath,amssymb,amsfonts,amsthm}
\usepackage{latexsym}
\usepackage{tikz}
\usepackage[normalem]{ulem}
\usetikzlibrary{automata,positioning}
\usetikzlibrary{chains,fit,shapes}
\usetikzlibrary{calc}
\usetikzlibrary{arrows}
\usepackage{float}
\usetikzlibrary{positioning,calc}
\usetikzlibrary{graphs}
\usetikzlibrary{graphs.standard}
\usetikzlibrary{arrows,decorations.markings}
\usepackage{multicol}
\usepackage{xcolor}
\usepackage{hyperref}
\usepackage{url}
\newtheorem{theorem}{Theorem}[section]

\newtheorem{lemma}{Lemma}[section]

\theoremstyle{definition}

\numberwithin{equation}{section}
\title{Word-Representability of Split Graphs with Independent Set of Size 4}

\author{\hspace{1cm} Suchanda Roy \ and Ramesh Hariharasubramanian \\ 
{{\footnotesize r.suchanda@iitg.ac.in},\ {\footnotesize  ramesh\_h@iitg.ac.in}}\\{\footnotesize Department of Mathematics, Indian Institute of Technology Guwahati, Guwahati, Assam 781039, India}}

\begin{document}
	\maketitle
	
	\begin{abstract}
 A pair of letters \( x \) and \( y \) are said to \textit{alternate} in a word \( w \) if, after removing all letters except for the copies of \(x\) and \(y\) from \( w \), the resulting word is of the form \( xyxy\ldots \) (of even or odd length) or \( yxyx\ldots \) (of even or odd length). A graph \( G = (V(G), E(G)) \) is \textit{word-representable} if there exists a word \( w \) over the alphabet \( V(G) \), such that any two distinct vertices \( x, y \in V(G)\) are adjacent in \( G \) (i.e., \( xy \in E(G) \)) if and only if the letters \( x \) and \( y \) alternate in \( w \). 

A split graph is a graph in which the vertices can be partitioned into a clique and an independent set. Word-representability of split graphs has been studied in a series of papers \cite{chen2022representing, iamthong2022word,kitaev2021word,kitaev2024semi} in the literature. In this work, we give a minimal forbidden induced subgraph characterization of word-representable split graphs with an independent set of size 4, which is an open problem posed by Kitaev and Pyatkin in \cite{kitaev2024semi} 
	\end{abstract}

	\maketitle
	\pagestyle{myheadings}
	
\section{Introduction}
 The theory of word-representable graphs, first introduced by Sergey Kitaev and Steven Seif in the setting of Perkins semigroups~\cite{kitaev2008word}, has become a rich area of research with strong connections to algebra, graph theory, and combinatorics on words. This graph class, which generalizes important classes such as circle graphs, comparability graphs, and 3-colorable graphs, has been extensively studied in monographs~\cite{kitaev2017comprehensive, kitaev2015words}. Computational tools like Glen's software have further supported their analysis and construction. Notably, recognizing whether a graph is word-representable or not is a NP-complete problem.

An orientation of a graph is \emph{transitive} if the presence of \( u \rightarrow v \) and \( v \rightarrow z \) implies \( u \rightarrow z \). Graphs that admit such an orientation are called comparability graphs.

A orientation of a graph is said to be \textit{semi-transitive} if it is acyclic and satisfies the following condition: for any directed path \( u_1 \rightarrow u_2 \rightarrow \cdots \rightarrow u_t \) with \( t \geq 2 \), either there is no edge from \( u_1 \) to \( u_t \), or the subgraph induced by the vertices \( u_1, \ldots, u_t \) forms a transitive tournament—that is, all possible edges \( u_i \rightarrow u_j \) exist for every \( i < j \). An undirected graph is said to be semi-transitive if it can be oriented in such a way that the resulting digraph is semi-transitive. This notion of semi-transitive orientation was introduced in~\cite{HALLDORSSON2016164} as a tool for characterizing \textit{word-representable graphs} (Theorem \ref{ch}).

A graph \( G = (V(G), E(G)) \) is said to be a \textit{split graph} if its vertex set \( V(G) \) can be partitioned into two disjoint subsets, where one subset induces a clique  and the other one induces an independent set~\cite{golumbic2004algorithmic}. Throughout this work, we assume that the clique is of maximum possible size, i.e., no vertex in the independent set is adjacent to all vertices in the clique.

Split graphs have received considerable attention in the literature, largely due to their rich structural properties and diverse applications (see, e.g.,~\cite{golumbic2004algorithmic} and references therein). A forbidden induced subgraph characterization for the class of split comparability graphs can be found in~\cite{golumbic2004algorithmic}. In the context of word-representability, it is known that while some split graphs admit semi-transitive orientations, others do not. The study of semi-transitive orientability, along with efforts to characterize split graphs via minimal forbidden subgraphs, has been the focus of several recent works~\cite{chen2022representing, iamthong2022word,kitaev2021word,kitaev2024semi} over the past 5--6 years. Although partial progress has been made in describing split graphs through forbidden subgraph characterizations, a complete classification remains open. In this work, we are studying a specific subclass of split graph, in particular those with an independent set of size 4.

In the following, Section \ref{p} presents existing results on the word-representability of split graphs and Section~\ref{s} presents our result on the characterization of word-representable split graphs with an independent set of size 4, in terms of minimal forbidden induced subgraphs, extending the existing body of work.
 \section{Preliminaries}\label{p}
One of the key developments in the study of word-representable graphs is their characterization in terms of semi-transitive orientations, which forms the basis for many subsequent results. We begin this section by stating this fundamental theorem. For a comprehensive account of the theory and related results, see~\cite{golumbic2004algorithmic,chen2022representing, iamthong2022word,kitaev2021word,kitaev2024semi}.

\begin{theorem}\cite{HALLDORSSON2016164}\label{ch} A graph is word-representable if and only if it admits a semi-transitive orientation.
\end{theorem}

The notion of semi-transitive orientation extends the classical concept of transitive orientation.

\begin{theorem}\cite{HALLDORSSON2016164}\label{co}
Let \( n \) be the number of vertices in a graph \( G \), and let \( x \in V(G) \) be a vertex of degree \( n - 1 \) (called an \emph{all-adjacent} vertex). Let \( H = G \setminus \{x\} \) be the graph obtained by removing \( x \) and all edges incident to it. Then, \( G \) is word-representable if and only if \( H \) is a comparability graph.
\end{theorem}

Let \( S_n = (E_{n-m}, K_m) \) be a split graph on \( n \) vertices, where the vertex set is partitioned into a maximal clique \( K_m \) on \( m \) vertices and an independent set \( E_{n-m} \) on \( n - m \) vertices. Each vertex in \( E_{n-m} \) has degree at most \( m - 1 \). In general, \(E\) denotes an independent set and \(K\) denotes a clique. We follow this notation throughout this work.

\begin{theorem} \cite{golumbic2004algorithmic}
Let \( S_n = (E_{n-m}, K_m) \) be a split graph. Then it is a comparability graph if and only if it does not contain an induced subgraph isomorphic to \( B_1 \), \( B_2 \), or \( B_3 \) (as shown in Figure \ref{fig:B123}).
\end{theorem}

\begin{figure}
\centering
\begin{tabular}{@{\hspace{0cm}}c@{\hspace{2cm}}c@{\hspace{2cm}}c@{\hspace{0cm}}}

\begin{tikzpicture}[scale=.8, every node/.style={circle, fill=black, inner sep=2pt}]
  \node (a) at (0,0) {};
  \node (b) at (1,0) {};
  \node (c) at (0,1) {};
  \node (d) at (1,1) {};
  \node (e) at (0.5,2) {};
  \node (f) at (0.5,3) {};
  \draw (a) -- (c);
  \draw (b) -- (d);
  \draw (d) -- (c);
  \draw (c) -- (e);
  \draw (e) -- (d);
  \draw (e) -- (f);
  \node[draw=none, fill=none] at (0.5,-0.5) {\textit{$B_1$}};
\end{tikzpicture}

&

\begin{tikzpicture}[scale=.8, every node/.style={circle, fill=black, inner sep=2pt}]
  \node (a) at (0,0) {};
  \node (b) at (1,0) {};
  \node (c) at (2,0) {};
  \node (d) at (0.5,1) {};
  \node (e) at (1.5,1) {};
  \node (f) at (1,2) {};
  \draw (a) -- (b) -- (c);
  \draw (a) -- (d) -- (f) -- (e) -- (c);
  \draw (d) -- (e);
  \draw (d) -- (b);
  \draw (b) -- (e);
  \node[draw=none, fill=none] at (1,-0.5) {\textit{$B_2$}};
\end{tikzpicture}

&

\begin{tikzpicture}[scale=.8, every node/.style={circle, fill=black, inner sep=2pt}]
  \node (a) at (0,0) {};
  \node (b) at (1,0) {};
  \node (c) at (0,1) {};
  \node (d) at (1,1) {};
  \node (e) at (-0.5,2) {};
  \node (f) at (0.5,2) {};
  \node (g) at (1.5,2) {};
  \draw (e) -- (f) -- (g);
  \draw (e) -- (c) -- (f);
  \draw (f) -- (d) -- (g);
  \draw (d) -- (c);
  \draw (d) -- (b);
  \draw (a) -- (c);
  \node[draw=none, fill=none] at (0.5,-0.5) {\textit{$B_3$}};
\end{tikzpicture}

\end{tabular}

 \vspace{0.2em}

\caption{Illustration of the graphs \( B_1 \), \( B_2 \) and \( B_3 \).}
\label{fig:B123}
\end{figure}
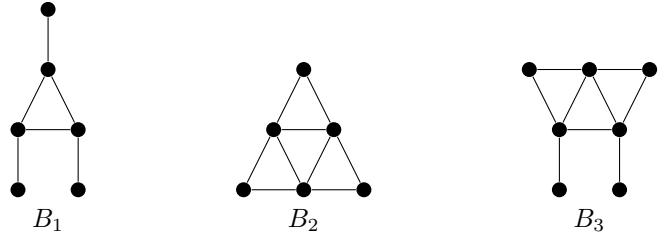
The following theorem highlights a notable property of transitive orientation when applied to cliques.

\begin{theorem}
\cite{kitaev2021word} Let \( K_m \) be a clique on \(m\) vertices in a graph \( G \). Then, any acyclic orientation of \( G \) induces a transitive orientation on \( K_m \). In particular, every semi-transitive orientation of \( G \) yields a transitive orientation on \( K_m \). Moreover, in both cases, the induced orientation on \( K_m \) contains exactly one source and one sink (as shown in Figure \ref{km}).
\end{theorem}

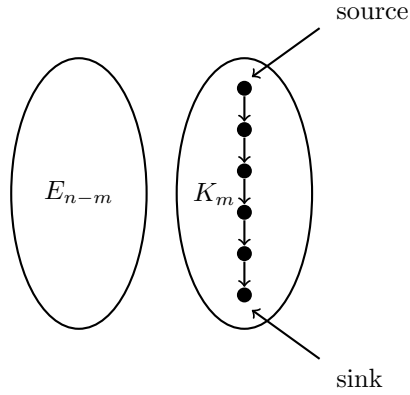
\begin{figure}[h]
\centering
\begin{tikzpicture}[scale=1]

    \def\ovalx{3.2}

    \draw[thick] (1,0) ellipse (0.9cm and 1.8cm);
    \node at (1,0) {\textit{$E_{n-m}$}};
    
    \draw[thick] (\ovalx,0) ellipse (0.9cm and 1.8cm);
    \node at (\ovalx - 0.4,0) {\textit{$K_m$}};

    \foreach \i in {0,...,5} {
        \node[circle, fill=black, inner sep=2pt] (v\i) at (\ovalx,1.4-\i*0.55) {};
    }

    \foreach \i in {0,...,4} {
        \pgfmathtruncatemacro{\j}{\i+1}
        \draw[->, thick] (v\i) -- (v\j);
    }

    \draw[->, thick] (\ovalx + 1, 2.2) -- (\ovalx + 0.1, 1.55); 
    \draw[->, thick] (\ovalx + 1, -2.2) -- (\ovalx + 0.1, -1.55); 

    \node at (\ovalx + 1.1, 2.2) [anchor=south west] {source};
    \node at (\ovalx + 1.1, -2.2) [anchor=north west] {sink};

\end{tikzpicture}
\caption{A schematic structure of a split graph with vertex partition into a clique $K_m$ and an independent set $E_{n-m}$. Arrows indicate directions and hierarchy among the vertices in $K_m$.}
\label{km}
\end{figure}

The following two theorems describe the structure of semi-transitive orientations in word-representable split graphs.

\begin{theorem}\label{th2}
\cite{kitaev2021word} Let \( S_n = (E_{n-m}, K_m) \) be a split graph, and let \( P = p_1 \rightarrow p_2 \rightarrow \cdots \rightarrow p_m \) be the longest directed path in \( K_m \). Then, any semi-transitive orientation of \( S_n \) partitions the vertices in the independent set \( E_{n-m} \) into three (possibly empty) categories:
\begin{itemize}

    \item A vertex is of \emph{type A} if it is a \textit{source} and is adjacent to all vertices in the set \( \{p_i, p_{i+1}, \ldots, p_j\} \) for some \( 1 \leq i \leq j \leq m \).

    \item A vertex is of \emph{type B} if it is a \textit{sink} and is adjacent to all vertices in the set \( \{p_i, p_{i+1}, \ldots, p_j\} \) for some \( 1 \leq i \leq j \leq m \).

    \item A vertex is of \emph{type C} if it has an incoming edge from every vertex in the set \( I_v = \{p_1, p_2, \ldots, p_i\} \) and an outgoing edge to every vertex in the set \( O_v = \{p_j, p_{j+1}, \ldots, p_m\} \), for some \( 1 \leq i < j \leq m \).

\end{itemize}
\end{theorem}

\begin{theorem}\label{th1}\cite{kitaev2021word} Let \( S_n = (E_{n-m}, K_m) \) be a graph with a semi-transitive orientation, where \( P = p_1 \rightarrow p_2 \rightarrow \cdots \rightarrow p_m \) is the longest directed path in \( K_m \). For a vertex \( x \in E_{n-m} \) of Type C, the following conditions must be satisfied:

\begin{itemize}

    \item There does not exist any vertex \( y \in I_{n-m} \) of Type A or Type B that is adjacent to both \( p_{|I_x|} \) and \( p_{m - |O_x| + 1} \).

    \item Furthermore, there is no vertex \( y \in I_{n-m} \) of Type C such that either its source group \( I_y \) or sink group \( O_y \) contains both \( p_{|I_x|} \) and \( p_{m - |O_x| + 1} \).

\end{itemize}

\end{theorem}

By combining the previous results, the semi-transitive orientations of split graphs can be fully classified as follows:

\begin{theorem}\cite{kitaev2021word} An orientation of a split graph \( S_n = (E_{n-m}, K_m) \) is semi-transitive if and only if the following conditions hold:

\begin{itemize}

    \item \( K_m \) is oriented transitively.

    \item Each vertex in \( E_{n-m} \) is of one of the three types described in Theorem~\ref{th2}.

    \item The restrictions in Theorem~\ref{th1} are satisfied.

\end{itemize}

\end{theorem}

Next, we state an equivalent formulation of the above theorem in terms of the neighborhood \( N(v) \) of a vertex \( v \in E \). 
\begin{theorem}
\cite{kitaev2024semi} \label{th} A split graph \(S\) is semi-transitive if and only if the vertices of the clique \(K\) can be labeled from 1 to \( k = |K| \) such that the following conditions hold:

\begin{itemize}
    \item For each \( v \in E \), the neighborhood \( N(v) \subseteq K \) is of the form \( [a, b] \) with \( a \leq b \), or \( [1, a] \cup [b, k] \) with \( a < b \).

    \item If \( N(u) = [a_1, b_1] \) and \( N(v) = [1, a_2] \cup [b_2, k] \), with \( a_1 \leq b_1 \) and \( a_2 < b_2 \), then either \( a_1 > a_2 \) or \( b_1 < b_2 \).

    \item If \( N(u) = [1, a_1] \cup [b_1, k] \) and \( N(v) = [1, a_2] \cup [b_2, k] \), with \( a_1 < b_1 \) and \( a_2 < b_2 \), then \( a_2 < b_1 \) and \( a_1 < b_2 \).
\end{itemize}

Here, for any integers \( a \leq b \), the notation \( [a, b] \) denotes the set \( \{a, a+1, \ldots, b\} \).
\end{theorem}

For another equivalent characterization based on the \emph{circular ones property}, we refer the reader to the relevant results presented in~\cite{kitaev2024semi}. Minimal forbidden induced subgraph characterizations of certain subclasses of split graphs have been established in the literature and form the foundation for the results presented in this work. As a complete characterization of split graphs in terms of minimal forbidden induced subgraphs remains unknown, researchers primarily focused on restricted subclasses to make progress. These restrictions typically include:
\begin{itemize}

    \item bounding the degree of vertices in the independent set \( E \),

    \item fixing the size of the clique \( K \), and

    \item fixing the size of the independent set \( E \).

\end{itemize}
\begin{theorem}\cite{kitaev2021word} Let \( S_n = (E_{n-m}, K_m) \) be a split graph with \( m \leq 3 \). Then \( S_n \) is word-representable.

\end{theorem}

\begin{lemma}\cite{kitaev2021word}\label{lemma} Let \( S_n = (E_{n-m}, K_m) \) be a split graph, and let \( S_{n+1} \) be a split graph obtained from \( S_n \) by one of the following operations:

\begin{itemize}

    \item Adding a vertex of degree 0 to \( E_{n-m} \),

    \item Adding a vertex of degree 1 to \( E_{n-m} \),

    \item Duplicating (i.e., copying) a vertex from either \( E_{n-m} \) or \( K_m \), where the new vertex has the same neighborhood as the original.

\end{itemize}

Then \( S_n \) is word-representable if and only if \( S_{n+1} \) is word-representable.

\end{lemma}

As a direct consequence of the previous lemma, we restrict our attention to the split graphs \( S_n = (E_{n-m}, K_m) \) satisfying the following properties:

\begin{itemize}

    \item No two vertices in \( S_n \) share the same neighborhood, up to adjacency between the vertices themselves.

    \item At most one vertex in the clique \( K_m \) is not adjacent to any vertex in the independent set \( E_{n-m} \).

    \item Every vertex in \( E_{n-m} \) has degree at least two.

\end{itemize}

These structural simplifications significantly reduce redundancy and assist in the development of the following characterization results corresponding to the types of restrictions discussed above.

\begin{theorem}\cite{kitaev2021word} Let \( m \geq 1 \) and \( S_n = (E_{n-m}, K_m) \) be a split graph, where each vertex in \( E_{n-m} \) has degree at most 2. Then \( S_n \) is word-representable if and only if it does not contain \( T_2 \) (as shown in Figure~\ref{fig:graph-configurations}) or any graph \( A_\ell \) (from Figure 6. of \cite{kitaev2021word}) as an induced subgraph.

\end{theorem}

\begin{theorem}\cite{kitaev2021word} Let \( S_n = (E_{n-4}, K_4) \) be a split graph. Then \( S_n \) is word-representable if and only if it does not contain \( T_1, T_2, T_3 , T_4 \) (as shown in Figure~\ref{fig:graph-configurations}) as induced subgraphs.

\end{theorem}

\begin{theorem}\cite{chen2022representing}\label{th5} Let \( S_n = (E_{n-5}, K_5) \) be a split graph. Then \( S_n \) is word-representable if and only if it does not contain \( T_1, T_2, \ldots, T_8 \) (as shown in Figure\ref{fig:graph-configurations}) and \( T_9 \) (from Figure 3 of \cite{chen2022representing}) as induced subgraphs.

\end{theorem}

\begin{theorem}\cite{kitaev2021word} Let \( S_n = (E_3, K_{n-3}) \) be a split graph. Then \( S_n \) is word-representable if and only if it does not contain \( T_1, T_2 \), and \( T_3 \) (as shown in Figure~\ref{fig:graph-configurations}) as induced subgraphs.

\end{theorem}

\section{Characterization of Word-Representable Split Graphs with a Fixed-Size Independent Set}\label{s} 

\begin{theorem}\label{m}Let \( S_n = (E_4, K_{n-4}) \) be a split graph. Then \( S_n \) is word-representable if and only if it does not contain any of the graphs \( T_1, T_2, \ldots, T_8 \ , D_1 \) (as shown in Figure~\ref{fig:graph-configurations}) as induced subgraphs.
\end{theorem}

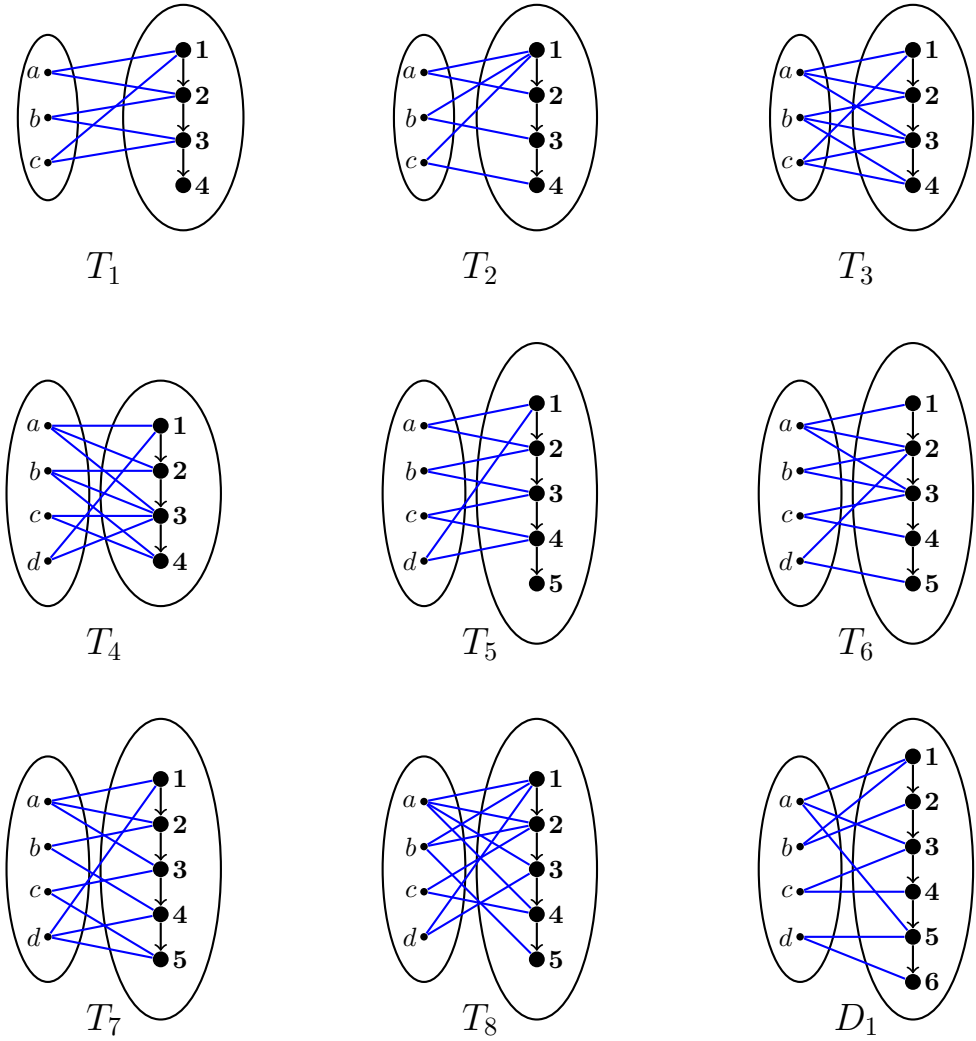
\begin{figure}
    \centering
    \begin{tikzpicture}[every node/.style={inner sep=1pt}, scale=1]
\tikzset{
  bigvtx/.style={circle, fill=black, minimum size=6pt, inner sep=0pt},
  cliqueedge/.style={black, thick, line width=0.5mm, -{Latex[length=1.5mm, width=1.5mm]}}
}

\def\xsep{5}
\def\ovalx{1.5}

\begin{scope}[shift={(0,0)}]
\draw[thick] (0,0) ellipse (.4cm and 1.1cm);
\node[circle, fill=black, label=left:$a$] (a1) at (0,0.6) {};
\node[circle, fill=black, label=left:$b$] (b1) at (0,0) {};
\node[circle, fill=black, label=left:$c$] (c1) at (0,-0.6) {};

\def\ovalx{1.8} 
\draw[thick] (\ovalx,0) ellipse (.8cm and 1.5cm);
\node[bigvtx, label=right:{\textbf{1}}] (x1) at (\ovalx,0.9) {};
\node[bigvtx, label=right:{\textbf{2}}] (y1) at (\ovalx,0.3) {};
\node[bigvtx, label=right:{\textbf{3}}] (z1) at (\ovalx,-0.3) {};
\node[bigvtx, label=right:{\textbf{4}}] (w1) at (\ovalx,-0.9) {};

\draw[black, thick, ->] (x1) -- (y1);
\draw[black, thick, ->] (y1) -- (z1);
\draw[black, thick, ->] (z1) -- (w1);

\draw[blue, thick] (a1) -- (x1);
\draw[blue, thick] (a1) -- (y1);
\draw[blue, thick] (b1) -- (y1);
\draw[blue, thick] (b1) -- (z1);
\draw[blue, thick] (c1) -- (z1);
\draw[blue, thick] (c1) -- (x1);

\node at (0.75,-2) {\Large $T_1$};
\end{scope}

\begin{scope}[shift={(\xsep,0)}]
\draw[thick] (0,0) ellipse (.4cm and 1.1cm);
\node[circle, fill=black, label=left:$a$] (a2) at (0,0.6) {};
\node[circle, fill=black, label=left:$b$] (b2) at (0,0) {};
\node[circle, fill=black, label=left:$c$] (c2) at (0,-0.6) {};

\draw[thick] (\ovalx,0) ellipse (.8cm and 1.5cm);
\node[bigvtx, label=right:{\textbf{1}}] (x2) at (\ovalx,0.9) {};
\node[bigvtx, label=right:{\textbf{2}}] (y2) at (\ovalx,0.3) {};
\node[bigvtx, label=right:{\textbf{3}}] (z2) at (\ovalx,-0.3) {};
\node[bigvtx, label=right:{\textbf{4}}] (w2) at (\ovalx,-0.9) {};

\draw[black, thick, ->] (x2) -- (y2);
\draw[black, thick, ->] (y2) -- (z2);
\draw[black, thick, ->] (z2) -- (w2);

\draw[blue, thick] (a2) -- (x2);
\draw[blue, thick] (a2) -- (y2);
\draw[blue, thick] (b2) -- (x2);
\draw[blue, thick] (b2) -- (z2);
\draw[blue, thick] (c2) -- (x2);
\draw[blue, thick] (c2) -- (w2);

\node at (0.75,-2) {\Large $T_2$};
\end{scope}

\begin{scope}[shift={(2*\xsep,0)}]
\draw[thick] (0,0) ellipse (.4cm and 1.1cm);
\node[circle, fill=black, label=left:$a$] (a3) at (0,0.6) {};
\node[circle, fill=black, label=left:$b$] (b3) at (0,0) {};
\node[circle, fill=black, label=left:$c$] (c3) at (0,-0.6) {};

\draw[thick] (\ovalx,0) ellipse (.8cm and 1.5cm);
\node[bigvtx, label=right:{\textbf{1}}] (x3) at (\ovalx,0.9) {};
\node[bigvtx, label=right:{\textbf{2}}] (y3) at (\ovalx,0.3) {};
\node[bigvtx, label=right:{\textbf{3}}] (z3) at (\ovalx,-0.3) {};
\node[bigvtx, label=right:{\textbf{4}}] (w3) at (\ovalx,-0.9) {};

\draw[black, thick, ->] (x3) -- (y3);
\draw[black, thick, ->] (y3) -- (z3);
\draw[black, thick, ->] (z3) -- (w3);

\draw[blue, thick] (a3) -- (x3);
\draw[blue, thick] (a3) -- (y3);
\draw[blue, thick] (a3) -- (z3);
\draw[blue, thick] (b3) -- (y3);
\draw[blue, thick] (b3) -- (z3);
\draw[blue, thick] (b3) -- (w3);
\draw[blue, thick] (c3) -- (z3);
\draw[blue, thick] (c3) -- (w3);
\draw[blue, thick] (c3) -- (x3);

\node at (0.75,-2) {\Large $T_3$};
\end{scope}


\begin{scope}[shift={(0,-5)}]
\draw[thick] (0,0) ellipse (.55cm and 1.5cm);
\node[circle, fill=black, label=left:$a$] (a4) at (0,0.9) {};
\node[circle, fill=black, label=left:$b$] (b4) at (0,0.3) {};
\node[circle, fill=black, label=left:$c$] (c4) at (0,-0.3) {};
\node[circle, fill=black, label=left:$d$] (d4) at (0,-0.9) {};

\draw[thick] (\ovalx,0) ellipse (.8cm and 1.5cm);
\node[bigvtx, label=right:{\textbf{1}}] (x4) at (\ovalx,0.9) {};
\node[bigvtx, label=right:{\textbf{2}}] (y4) at (\ovalx,0.3) {};
\node[bigvtx, label=right:{\textbf{3}}] (z4) at (\ovalx,-0.3) {};
\node[bigvtx, label=right:{\textbf{4}}] (w4) at (\ovalx,-0.9) {};

\draw[black, thick, ->] (x4) -- (y4);
\draw[black, thick, ->] (y4) -- (z4);
\draw[black, thick, ->] (z4) -- (w4);

\draw[blue, thick] (a4) -- (x4);
\draw[blue, thick] (a4) -- (y4);
\draw[blue, thick] (a4) -- (z4);
\draw[blue, thick] (b4) -- (y4);
\draw[blue, thick] (b4) -- (z4);
\draw[blue, thick] (b4) -- (w4);
\draw[blue, thick] (c4) -- (z4);
\draw[blue, thick] (c4) -- (w4);
\draw[blue, thick] (d4) -- (z4);
\draw[blue, thick] (d4) -- (x4);

\node at (0.75,-2) {\Large $T_4$};
\end{scope}

\begin{scope}[shift={(\xsep,-5)}]
\draw[thick] (0,0) ellipse (.55cm and 1.5cm);
\node[circle, fill=black, label=left:$a$] (a5) at (0,0.9) {};
\node[circle, fill=black, label=left:$b$] (b5) at (0,0.3) {};
\node[circle, fill=black, label=left:$c$] (c5) at (0,-0.3) {};
\node[circle, fill=black, label=left:$d$] (d5) at (0,-0.9) {};

\draw[thick] (\ovalx,0) ellipse (.8cm and 2cm);
\node[bigvtx, label=right:{\textbf{1}}] (v5) at (\ovalx,1.2) {};
\node[bigvtx, label=right:{\textbf{2}}] (x5) at (\ovalx,0.6) {};
\node[bigvtx, label=right:{\textbf{3}}] (y5) at (\ovalx,0) {};
\node[bigvtx, label=right:{\textbf{4}}] (z5) at (\ovalx,-0.6) {};
\node[bigvtx, label=right:{\textbf{5}}] (w5) at (\ovalx,-1.2) {};

\draw[black, thick, ->] (v5) -- (x5);
\draw[black, thick, ->] (x5) -- (y5);
\draw[black, thick, ->] (y5) -- (z5);
\draw[black, thick, ->] (z5) -- (w5);

\draw[blue, thick] (a5) -- (v5);
\draw[blue, thick] (a5) -- (x5);
\draw[blue, thick] (b5) -- (x5);
\draw[blue, thick] (b5) -- (y5);
\draw[blue, thick] (c5) -- (y5);
\draw[blue, thick] (c5) -- (z5);
\draw[blue, thick] (d5) -- (z5);
\draw[blue, thick] (d5) -- (v5);

\node at (0.75,-2) {\Large $T_5$};
\end{scope}

\begin{scope}[shift={(2*\xsep,-5)}]
\draw[thick] (0,0) ellipse (.55cm and 1.5cm);
\node[circle, fill=black, label=left:$a$] (a5) at (0,0.9) {};
\node[circle, fill=black, label=left:$b$] (b5) at (0,0.3) {};
\node[circle, fill=black, label=left:$c$] (c5) at (0,-0.3) {};
\node[circle, fill=black, label=left:$d$] (d5) at (0,-0.9) {};

\draw[thick] (\ovalx,0) ellipse (.8cm and 2cm);
\node[bigvtx, label=right:{\textbf{1}}] (v5) at (\ovalx,1.2) {};
\node[bigvtx, label=right:{\textbf{2}}] (x5) at (\ovalx,0.6) {};
\node[bigvtx, label=right:{\textbf{3}}] (y5) at (\ovalx,0) {};
\node[bigvtx, label=right:{\textbf{4}}] (z5) at (\ovalx,-0.6) {};
\node[bigvtx, label=right:{\textbf{5}}] (w5) at (\ovalx,-1.2) {};

\draw[black, thick, ->] (v5) -- (x5);
\draw[black, thick, ->] (x5) -- (y5);
\draw[black, thick, ->] (y5) -- (z5);
\draw[black, thick, ->] (z5) -- (w5);

\draw[blue, thick] (a5) -- (v5);
\draw[blue, thick] (a5) -- (x5);
\draw[blue, thick] (a5) -- (y5);
\draw[blue, thick] (b5) -- (x5);
\draw[blue, thick] (b5) -- (y5);
\draw[blue, thick] (c5) -- (y5);
\draw[blue, thick] (c5) -- (z5);
\draw[blue, thick] (d5) -- (w5);
\draw[blue, thick] (d5) -- (x5);

\node at (0.75,-2) {\Large $T_6$};
\end{scope}

\begin{scope}[shift={(0,-10)}]
\draw[thick] (0,0) ellipse (.55cm and 1.5cm);
\node[circle, fill=black, label=left:$a$] (a5) at (0,0.9) {};
\node[circle, fill=black, label=left:$b$] (b5) at (0,0.3) {};
\node[circle, fill=black, label=left:$c$] (c5) at (0,-0.3) {};
\node[circle, fill=black, label=left:$d$] (d5) at (0,-0.9) {};

\draw[thick] (\ovalx,0) ellipse (.8cm and 2cm);
\node[bigvtx, label=right:{\textbf{1}}] (v5) at (\ovalx,1.2) {};
\node[bigvtx, label=right:{\textbf{2}}] (x5) at (\ovalx,0.6) {};
\node[bigvtx, label=right:{\textbf{3}}] (y5) at (\ovalx,0) {};
\node[bigvtx, label=right:{\textbf{4}}] (z5) at (\ovalx,-0.6) {};
\node[bigvtx, label=right:{\textbf{5}}] (w5) at (\ovalx,-1.2) {};

\draw[black, thick, ->] (v5) -- (x5);
\draw[black, thick, ->] (x5) -- (y5);
\draw[black, thick, ->] (y5) -- (z5);
\draw[black, thick, ->] (z5) -- (w5);

\draw[blue, thick] (a5) -- (v5);
\draw[blue, thick] (a5) -- (x5);
\draw[blue, thick] (a5) -- (y5);
\draw[blue, thick] (b5) -- (x5);
\draw[blue, thick] (b5) -- (z5);
\draw[blue, thick] (c5) -- (y5);
\draw[blue, thick] (c5) -- (w5);
\draw[blue, thick] (d5) -- (z5);
\draw[blue, thick] (d5) -- (w5);
\draw[blue, thick] (d5) -- (v5);

\node at (0.75,-2) {\Large $T_7$};
\end{scope}

\begin{scope}[shift={(\xsep,-10)}]
\draw[thick] (0,0) ellipse (.55cm and 1.5cm);
\node[circle, fill=black, label=left:$a$] (a5) at (0,0.9) {};
\node[circle, fill=black, label=left:$b$] (b5) at (0,0.3) {};
\node[circle, fill=black, label=left:$c$] (c5) at (0,-0.3) {};
\node[circle, fill=black, label=left:$d$] (d5) at (0,-0.9) {};

\draw[thick] (\ovalx,0) ellipse (.8cm and 2cm);
\node[bigvtx, label=right:{\textbf{1}}] (v5) at (\ovalx,1.2) {};
\node[bigvtx, label=right:{\textbf{2}}] (x5) at (\ovalx,0.6) {};
\node[bigvtx, label=right:{\textbf{3}}] (y5) at (\ovalx,0) {};
\node[bigvtx, label=right:{\textbf{4}}] (z5) at (\ovalx,-0.6) {};
\node[bigvtx, label=right:{\textbf{5}}] (w5) at (\ovalx,-1.2) {};

\draw[black, thick, ->] (v5) -- (x5);
\draw[black, thick, ->] (x5) -- (y5);
\draw[black, thick, ->] (y5) -- (z5);
\draw[black, thick, ->] (z5) -- (w5);

\draw[blue, thick] (a5) -- (v5);
\draw[blue, thick] (a5) -- (x5);
\draw[blue, thick] (a5) -- (y5);
\draw[blue, thick] (a5) -- (z5);
\draw[blue, thick] (b5) -- (v5);
\draw[blue, thick] (b5) -- (x5);
\draw[blue, thick] (b5) -- (w5);
\draw[blue, thick] (c5) -- (x5);
\draw[blue, thick] (c5) -- (z5);
\draw[blue, thick] (d5) -- (v5);
\draw[blue, thick] (d5) -- (y5);

\node at (0.75,-2) {\Large $T_8$};
\end{scope}

\begin{scope}[shift={(2*\xsep,-10)}]
\draw[thick] (0,0) ellipse (.55cm and 1.5cm);
\node[circle, fill=black, label=left:$a$] (a5) at (0,0.9) {};
\node[circle, fill=black, label=left:$b$] (b5) at (0,0.3) {};
\node[circle, fill=black, label=left:$c$] (c5) at (0,-0.3) {};
\node[circle, fill=black, label=left:$d$] (d5) at (0,-0.9) {};

\draw[thick] (\ovalx,0) ellipse (.8cm and 2cm);
\node[bigvtx, label=right:{\textbf{1}}] (u5) at (\ovalx,1.5) {};
\node[bigvtx, label=right:{\textbf{2}}] (v5) at (\ovalx,0.9) {};
\node[bigvtx, label=right:{\textbf{3}}] (w5) at (\ovalx,0.3) {};
\node[bigvtx, label=right:{\textbf{4}}] (x5) at (\ovalx,-0.3) {};
\node[bigvtx, label=right:{\textbf{5}}] (y5) at (\ovalx,-0.9) {};
\node[bigvtx, label=right:{\textbf{6}}] (z5) at (\ovalx,-1.5) {};

\draw[black, thick, ->] (u5) -- (v5);
\draw[black, thick, ->] (v5) -- (w5);
\draw[black, thick, ->] (w5) -- (x5);
\draw[black, thick, ->] (x5) -- (y5);
\draw[black, thick, ->] (y5) -- (z5);

\draw[blue, thick] (a5) -- (u5);
\draw[blue, thick] (a5) -- (w5);
\draw[blue, thick] (a5) -- (y5);
\draw[blue, thick] (b5) -- (u5);
\draw[blue, thick] (b5) -- (v5);
\draw[blue, thick] (c5) -- (w5);
\draw[blue, thick] (c5) -- (x5);
\draw[blue, thick] (d5) -- (y5);
\draw[blue, thick] (d5) -- (z5);

\node at (0.75,-2) {\Large $D_1$};
\end{scope}

\end{tikzpicture}
  \caption{Illustration of the graphs from $T_1$ to $D_1$, depicting various configurations. Each right oval schematically represents a clique, similar to the one shown in Figure~\ref{km}.}
  \label{fig:graph-configurations}
  \end{figure}

To list down all minimal forbidden induced subgraphs for our problem, we primarily rely on Lemma \ref{lemma} and Theorem \ref{th}. Our method involves a systematic partitioning of all possible cases to ensure comprehensive coverage of the problem space. For each case, we identify the possible maximal word-representable as well as semi-transitive split graphs \( S_n \) relevant to the given context. Subsequently, by introducing an additional vertex, we derive the corresponding minimal non-word-representable graphs.

\begin{lemma}
The split graph \( D_1\) shown in Figure \ref{fig:graph-configurations}, is a minimal non-word-representable graph.

\end{lemma}

\begin{proof}
To demonstrate the minimality of this graph, we show that deleting any vertex yields a word-representable graph. Considering only non-isomorphic cases, we provide vertex labeling in each case that satisfy the conditions of Theorem \ref{th}.
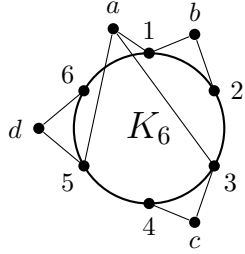
\begin{figure}
\centering
\begin{tikzpicture}[scale=1, every node/.style={circle, fill=black, inner sep=1.5pt}]

\coordinate (1) at (90:1);    
\coordinate (2) at (30:1);    
\coordinate (3) at (-30:1);   
\coordinate (4) at (-90:1);   
\coordinate (5) at (210:1);   
\coordinate (6) at (150:1);   
\draw[thick] (0,0) circle (1);

\path (1) -- (2) coordinate[pos=0.7] (b');
\path (3) -- (4) coordinate[pos=0.3] (c');
\path (5) -- (6) coordinate[pos=0.5] (d');
\path (6) -- (1) coordinate[pos=0.45] (a');

\node[draw=none, fill=none] at (0,0) {\Large $K_6$};

\node[label=above:$b$] (b) at ($(b') + (0,0.6)$) {};
\node[label=below:$c$] (c) at ($(c') + (0,-0.6)$) {};
\node[label=left:$d$] (d) at ($(d') + (-0.6,0)$) {};
\node[label=above:$a$] (a) at ($(a') + (0,0.6)$) {};

\node[label=above:$1$] at (1) {};
\node[label=right:$2$] at (2) {};
\node[label=below right:$3$] at (3) {};
\node[label=below:$4$] at (4) {};
\node[label=below left:$5$] at (5) {};
\node[label=above left:$6$] at (6) {};

\draw (b) -- (1);
\draw (b) -- (2);

\draw (c) -- (3);
\draw (c) -- (4);

\draw (d) -- (5);
\draw (d) -- (6);

\draw (a) -- (3);
\draw (a) -- (1);
\draw (a) -- (5);

\end{tikzpicture}
\caption{A schematic representation of \(D_1\), where the clique vertices lie on the circle and the independent vertices are positioned outside the circle.
}
\label{D_1}
\end{figure}
    \begin{itemize}

        \item \underline{\textit{Deletion of vertex \(a\):}}\quad \( E \): \(b>c>d\),\quad\( K \): \(1>2>3>4>5>6 \).

        \item \underline{\textit{Deletion of vertex \(b\):}}\quad \( E \): \(d>a>c\),\quad\( K \): \(6>5>1>3>4>2 \).

        \item \underline{\textit{Deletion of vertex \(1\):}}\quad \( E \): \(d>a>c>b\),\quad\( K \): \(6>5>3>4>2\).

        \item \underline {\textit{Deletion of vertex \(2\):}}\quad \( E \): \(c>a>b>d\),\quad\( K \): \(4>3>1>5>6 \).\\

    \end{itemize}
 
To prove that \( D_1 \) is not word-representable, we proceed by contradiction. Assume that \( D_1 \) is word-representable. According to Theorem~\ref{th}, the neighborhood of every vertex of degree 2 must consist of two consecutive vertices on the circle, as illustrated in Figure~\ref{D_1}. Under this assumption, the neighborhood of the vertex \( a \) of degree 3 is forced to be non-consecutive vertices on the circle. This contradicts Theorem~\ref{th}, and therefore \( D_1 \) is not word-representable.

  \end{proof}

For the minimality and non-word-representability of \( T_1, T_2, \ldots, T_8 \); we refer the reader to \cite{chen2022representing},\cite{kitaev2021word}.

Before proceeding to the proof, we introduce some notations to simplify the write-up.

\begin{itemize}

    \item Let \( a, b, c, d \) be the vertices of the independent set \( E \).

    \item For any subset \( S \subseteq \{a, b, c, d\} \), the notation \( v_S \) denotes a vertex in the clique whose neighborhood in \( E \) is exactly \( S \).

    \item For a vertex \( v \) in the clique, we define its degree with respect to the independent set \( E \) as
\[
d_E(v) = |N(v) \cap E|,
\]
where \( N(v) \) denotes the neighborhood of \( v \) in the graph.

     \item Let \( x, y, z, w \in \{a, b, c, d\} \), where \( \{a, b, c, d\} \) denotes the independent set \( E \).

     \item Let ordering $\mathcal{O}$ be \( x > y > z > w \), where \( x, y, z, w \in \{a, b, c, d\} = E \).

\end{itemize}

\begin{proof}[Proof of Theorem 15]

To establish that the only minimal forbidden subgraphs in all possible cases of a split graph with \( |E| = 4 \) are \( T_1, T_2, \ldots, T_{8} \), together with \( D_1 \), we analyze the following cases. For each case, we consider only non-isomorphic configurations among its subcases. To prove word-representability, we apply Theorem~\ref{th} using the specified vertex ordering.

\begin{enumerate}

    \item \textit{There does not exist any vertex \( v \in K \) such that \( d_E(v) = 3 \) or \( d_E(v) = 4 \).}

    \begin{enumerate}

        \item \textit{There exist exactly two vertices, say \( v_1 \) and \( v_2 \), in \( K \) such that
        \[
        d_E(v_1) = d_E(v_2) = 2.\]}

        \begin{enumerate}

            \item \( v_1 = v_{\{x,y\}} \), \( v_2 = v_{\{y,z\}} \)\\
            To prove semi-transitivity of all possible split graphs here, we consider \( E \) with ordering $\mathcal{O}$ and \( K \) in the following order:

            \[
            v_{\{x\}} > v_{\{x,y\}} > v_{\{y\}} > v_{\{y,z\}} > v_{\{z\}} > v_{\{w\}} > v_{\emptyset}
            \]

            \item \( v_1 = v_{\{x,y\}} \), \( v_2 = v_{\{z,w\}} \)\\
            To prove semi-transitivity of all possible split graphs under this case, consider \( E \) with ordering $\mathcal{O}$ and \( K \) with the following order:

            \[
            v_{\{x\}} > v_{\{x,y\}} > v_{\{y\}} > v_{\{z\}} > v_{\{z,w\}} > v_{\{w\}} > v_{\emptyset}
            \]

        \end{enumerate}

        \item \textit{There exist exactly three vertices, say \( v_1 \), \( v_2 \) and \( v_3 \) in \( K \) such that
        \[
        d_E(v_1) = d_E(v_2) = d_E(v_3) = 2.
        \]}

        \begin{enumerate}

            \item {\( v_1 = v_{\{x,y\}} \), \( v_2 = v_{\{y,z\}} \), \(v_3 = v_{\{x,z\} } \)}

            \begin{itemize}

                \item Addition of \(v_{\emptyset}\) or \(v_{w}\) yields a graph containing \(T_1\) as an induced subgraph.
\\For other cases,

                \item  \( E \): $\mathcal{O}$, \( K \): \(v_{\{x,z\}}> v_{\{x\}}>  v_{\{x,y\}}>v_{\{y\}}> v_{\{y,z\}}> v_{\{z\}}\).

            \end{itemize}

            \item  {\( v_1 = v_{\{x,y\}} \), \( v_2 = v_{\{x,z\}} \), \(v_3 = v_{\{x,w\} } \)}

            \begin{itemize}

                \item Addition of  \(v_{y},v_{z},v_{w}\) yields a graph containing \(D_1\) as an induced subgraph.\\
                 For other cases, 

                \item \( E \): \(y>x>z>w\), \( K \): \(v_{\{y\}}> v_{\{x,y\}}>  v_{\{x\}}>v_{\{x,z\}}> v_{\{x,w\}}> v_{\{w\}}>v_{\emptyset}\).

            \end{itemize}

            \item  {\( v_1 = v_{\{x,y\}} \), \( v_2 = v_{\{y,z\}} \), \(v_3 = v_{\{x,w\} } \)}

           To prove semi-transitivity of all possible split graphs here, we consider  
            \( E \): \(w>x>y>z\), \( K \):  \(v_{\{w\}}> v_{\{x,w\}}>  v_{\{x\}}>v_{\{x,y\}}> v_{\{y\}}>v_{\{y,z\}}> v_{\{z\}}> v_{\emptyset}\).

        \end{enumerate}

         \item {\textit{There exist exactly four vertices, say \( v_1 \), \( v_2 \), \(v_3\) and  \(v_4\) in \( K \) such that
\[
d_E(v_1) = d_E(v_2) = d_E(v_3) = d_E(v_4) = 2.
\]}}

        \begin{enumerate}

            \item {\( v_1 = v_{\{x,y\}} \), \( v_2 = v_{\{y,z\}} \), \( v_3 = v_{\{z,w\}} \), \(v_4 = v_{\{x,w\} } \)}\label{1.3.1}

            \begin{itemize}

                \item Addition of \(v_{\emptyset}\) yields a graph containing \(T_5\) as an induced subgraph.\\For other cases, 

                \item  \( E \): $\mathcal{O}$, \( K \): \(v_{\{x,w\}}> v_{\{x\}}>  v_{\{x,y\}}>v_{\{y\}}> v_{\{y,z\}}> v_{\{z\}}> v_{\{z,w\}}> v_{\{w\}}\).

            \end{itemize}

            \item {\( v_1 = v_{\{x,y\}} \), \( v_2 = v_{\{x,z\}} \), \( v_3 = v_{\{x,w\}} \), \(v_4 = v_{\{y,z\} } \)}

            \begin{itemize}

                \item Addition of \(v_{\emptyset}\) or  \(v_{w}\) yields a graph containing \(T_1\) as an induced subgraph.\\For other cases, 

                \item  \( E \): \(y>x>w>z\), \( K \): \(v_{\{y,z\}}> v_{\{y\}}>  v_{\{x,y\}}>v_{\{x\}}> v_{\{x,w\}}> v_{\{x,z\}}>v_{\{z\}}\).

            \end{itemize}

        \end{enumerate}

         \item {\textit{There exist atleast five vertices, say \( v_1 \), \( v_2 \), \(v_3\), \(v_4\) and \(v_5\)  in \( K \) such that
\[
d_E(v_1) = d_E(v_2) = d_E(v_3) = d_E(v_4) = d_E(v_5) = 2.\]}}\\
Here, consideration of only one scenario is enough; the rest of the things are isomorphic to it. Let \( v_1 = v_{\{x,y\}}\), \(v_2 = v_{\{x,z\}}\), \(v_3 = v_{\{x,w\}}\),  \(v_4 = v_{\{y,z\}}\), \(v_5 = v_{\{y,w\}} \). The graph itself in this case contains \(T_7\) as an induced subgraph. Hence, we don't need to consider subcases of it. Also, this case is discarded from the further extensions in the next cases.

    \end{enumerate}

    \item {\textit{There exists exactly one vertex \( v_1\in K \) such that \( d_E(v_1) = 3 \), and there is no vertex in \( K \) such that  \( d_E(v) = 4 \).}}\\

   Without loss of generality, let \( v_1 = v_{\{x,y,z\}} \). In this case, if we add the vertices \( v_{\{x\}}, v_{\{y\}}, v_{\{z\}} \) to the clique \( K \) with the existing graph structure, the resulting graph will contain \( T_2 \) as an induced subgraph. Therefore, we shall exclude this scenario from consideration in the subsequent subcases while extending the graph.

    \begin{enumerate}

        \item {\textit{{There exists exactly one vertex, say \( v_1 \) in \( K \) such that
\[ d_E(v_2) = 2.
\]
}}}

         \begin{enumerate}

            \item {\( v_2 = v_{\{x,y\}} \)}\\
            To cover all possible split graphs in this case, we give the vertex labelings of two non-isomorphic graphs.

           \begin{itemize}

    \item \( E\): $\mathcal{O}$,  \(K\):
    \(v_{\{x\}} > v_{\{x,y,z\}} > v_{\{x,y\}} > v_{\{y\}} > v_{\{w\}} > v_{\emptyset}. \)
    \item \( E\): $\mathcal{O}$, \(K:
    v_{\{x\}}  > v_{\{x,y\}} > v_{\{x,y,z\}} > v_{\{z\}} > v_{\{w\}} > v_{\emptyset}. \)
\end{itemize}

            \item {\( v_2 = v_{\{z,w\}} \)}\label{2.1.2}

            \begin{itemize}

    \item Addition of \( v_{\{x\}} \), \( v_{\{y\}} \) yields a graph containing \( T_2 \) as an induced subgraph.\\
      For the remaining cases,

    \item \( E\): $\mathcal{O}$, \( K \): \( v_{\{x\}}> v_{\{x,y,z\}}> v_{\{z\}}> v_{\{z,w\}}> v_{\{w\}}> v_{\emptyset} \).

\end{itemize}

        \end{enumerate}

        \item {\textit{There exist exactly two vertices, say \( v_2 \) and \( v_3 \), in \( K \) such that
\[
d_E(v_2) = d_E(v_3) = 2.
\]
}}

        \begin{enumerate} 

            \item {\( v_2 = v_{\{x,y\}} \), \( v_3 = v_{\{y,z\}} \)}\label{2.2.1}\\
             To cover all possible split graphs in this case, we give the vertex labelings of two non-isomorphic graphs.

           \begin{itemize}

    \item \( E\): $\mathcal{O}$, \( K:
    v_{\{x\}} > v_{\{x,y\}}> v_{\{x,y,z\}} > v_{\{y,z\}} > v_{\{z\}} > v_{\{w\}} > v_{\emptyset}. \)
    \item \( E\): $\mathcal{O}$, \( K:
    v_{\{x\}}  > v_{\{x,y\}} > v_{\{x,y,z\}}> v_{\{y,z\}} > v_{\{y\}} > v_{\{w\}} > v_{\emptyset}. \)
\end{itemize}
            
            \item {\( v_2 = v_{\{x,y\}} \), \( v_3 = v_{\{z,w\}} \)}\label{2.2.2}

            \begin{itemize}

    \item The case of adding \( v_{\{x\}} \), \( v_{\{y\}} \) is already discussed in \ref{2.1.2}.\\
    For the remaining cases,

    \item \( E \): $\mathcal{O}$, \( K \): \( v_{\{x\}}>  v_{\{x,y\}}>v_{\{x,y,z\}}> v_{\{z\}}> v_{\{z,w\}}> v_{\{w\}}> v_{\emptyset} \).

\end{itemize}

            \item {\( v_2 = v_{\{x,y\}} \), \( v_3 = v_{\{y,w\}} \)}

            \begin{itemize}

    \item Addition of \( v_{\{x\}} \), \( v_{\{z\}} \) yields a graph containing \(T_2\) as an induced subgraph.\\
    For the remaining cases, we provide vertex labelings of two non-isomorphic graphs.

    \item  \( E \): $\mathcal{O}$, \( K \): \( v_{\{x\}}>  v_{\{x,y\}}>v_{\{x,y,z\}}> v_{\{y\}}> v_{\{y,w\}}> v_{\{w\}}> v_{\emptyset} \). 

    \item  \( E \): \(w > y > x > z\), \( K \): \( v_{\{w\}}>  v_{\{w,y\}}>v_{\{y\}}> v_{\{y,x\}}>v_{\{x,y,z\}}> v_{\{z\}}> v_{\emptyset} \).

\end{itemize}

            \item {\( v_2 = v_{\{z,w\}} \), \( v_3 = v_{\{y,w\}} \).}\label{2.2.4}

            \begin{itemize}

    \item Addition of \( v_{\{x\}} \)  yields a graph containing \(T_2\) as an induced subgraph.

    \item Addition of \(v_{\emptyset} \) yields a graph containing \(T_1\) as an induced subgraph.\\
     For the remaining cases,

    \item \( E \): \(x> y > w > z\), \( K \): \(v_{\{x,y,z\}}> v_{\{y\}}>  v_{\{y,w\}}>v_{\{w\}}> v_{\{w,z\}}> v_{\{z\}}\).

\end{itemize}

        \end{enumerate}

         \item {\textit{There exist exactly three vertices, say \( v_2 \) , \( v_3 \) and \(v_4\) in \( K \) such that
\[
d_E(v_2) = d_E(v_3) = d_E(v_4) = 2.
\]}}

        \begin{enumerate}

            \item {\( v_2 = v_{\{x,y\}} \), \( v_3 = v_{\{y,z\}} \), \( v_4 = v_{\{x,z\}} \) }\label{2.3.1}\\
            This graph is itself \(T_3\).

            \item {\( v_2 = v_{\{x,y\}} \), \( v_3 = v_{\{y,w\}} \), \( v_4 = v_{\{x,w\}} \) }

             \begin{itemize}

    \item Addition of \( v_{\{z\}} \)  yields a graph containing \(T_2\) as an induced subgraph.

    \item Addition of \(v_{\emptyset} \) yields a graph containing \(T_1\) as an induced subgraph.\\
    For the remaining cases,

    \item \( E \): \(x> z > y> w\), \( K \): \(v_{\{x,w\}}> v_{\{x\}}>  v_{\{x,y,z\}}>v_{\{x,y\}}> v_{\{y\}}> v_{\{y,w\}}> v_{\{w\}}\).

\end{itemize}

            \item {\( v_2 = v_{\{x,y\}} \), \( v_3 = v_{\{x,z\}} \), \( v_4 = v_{\{x,w\}} \) }

            \begin{itemize}

    \item Addition of \( v_{\{z\}}, v_{\{y\}} \) yields a graph containing \(T_2\) as induced subgraph.\\
    For the remaining cases,

    \item  \( E \): \(y> z > x> w\), \( K \): \(v_{\{y\}}> v_{\{x,y\}}>  v_{\{x,y,z\}}>v_{\{x,z\}}> v_{\{x\}}> v_{\{x,w\}}> v_{\{w\}}> v_{\emptyset}\).

\end{itemize}

            \item {\( v_2 = v_{\{x,w\}} \), \( v_3 = v_{\{y,w\}} \), \( v_4 = v_{\{z,w\}} \) }\label{2.3.4}\\
            This graph contains \(T_2\).

            \item {\( v_2 = v_{\{x,y\}} \), \( v_3 = v_{\{y,z\}} \), \( v_4 = v_{\{x,w\}} \) }

            \begin{itemize}

    \item Adding \( v_{\{z\}}, v_{\{y\}} \)  yields a graph containing \(T_2\) as an induced subgraph.\\
    For the remaining cases, we provide vertex labelings of two non-isomorphic graphs.

    \item  \( E \): \(w> x > y> z\), \( K \): \(v_{\{w\}}> v_{\{x,w\}}> v_{\{x\}}>  v_{\{x,y\}}>v_{\{x,y,z\}}> v_{\{y,z\}}> v_{\{z\}}> v_{\emptyset}\). 

    \item  \( E \): \(w> x > y> z\), \( K \): \(v_{\{w\}}> v_{\{x,w\}}> v_{\{x\}}>  v_{\{x,y\}}>v_{\{x,y,z\}}> v_{\{y,z\}}> v_{\{y\}}> v_{\emptyset}\).

\end{itemize}

            \item {\( v_2 = v_{\{x,y\}} \), \( v_3 = v_{\{z,w\}} \), \( v_4 = v_{\{x,w\}} \) }\label{2.3.6}

            \begin{itemize}

    \item Addition of \( v_{\{y\}} \)  yields a graph containing \(T_2\) as an induced subgraph.

    \item Addition of \(v_{\emptyset} \) yields a graph containing \(T_1\) as an induced subgraph.\\
    For the remaining cases,

    \item \( E \): $\mathcal{O}$, \( K \): \(v_{\{x,w\}}> v_{\{x\}}>  v_{\{x,y\}}>v_{\{x,y,z\}}> v_{\{z\}}> v_{\{z,w\}}> v_{\{w\}}\).

\end{itemize}

        \end{enumerate}

         \item {\textit{There exist exactly four vertices, say \( v_2 \) , \( v_3 \) , \(v_4\) and  \(v_5\) in \( K \) such that
\[
d_E(v_2) = d_E(v_3) = d_E(v_4) = d_E(v_5) = 2.
\]}}

       \begin{enumerate}

            \item {\( v_2 = v_{\{x,y\}} \), \( v_3 = v_{\{y,z\}} \), \( v_4 = v_{\{z,w\}} \), \(v_5 = v_{\{x,w\} } \)}

            \begin{itemize}

            \item Addition of \( v_{\{y\}} \) yields a graph containing \(T_2\) as an induced subgraph.

    \item The case of adding \(v_{\emptyset} \) is already discussed in \ref{1.3.1}.\\
    For the remaining cases,
    \item \( E \): $\mathcal{O}$, \( K \): \(v_{\{x,w\}}> v_{\{x\}}>  v_{\{x,y\}}>  v_{\{x,y,z\}}>v_{\{y,z\}}> v_{\{z\}}> v_{\{z,w\}}> v_{\{w\}}\).

    \end{itemize}

            \item {\( v_2 = v_{\{x,y\}} \), \( v_3 = v_{\{x,z\}} \), \( v_4 = v_{\{x,w\}} \), \(v_5 = v_{\{y,z\} } \)}\\
            This case is already considered in \ref{2.3.1}.

            \item {\( v_2 = v_{\{x,y\}} \), \( v_3 = v_{\{x,z\}} \), \( v_4 = v_{\{x,w\}} \), \(v_5 = v_{\{z,w\} } \)}
            \begin{itemize}
    \item The case of adding \( v_{\{y\}} \) is already considered in \ref{2.3.6}.

    \item The case of adding \(v_{\emptyset} \) is already considered in \ref{2.3.6}.\\
    For the remaining cases,

    \item  \( E \): $\mathcal{O}$, \( K \): \(v_{\{x,w\}}> v_{\{x\}}>  v_{\{x,y\}}>v_{\{x,y,z\}}>v_{\{x,z\}}> v_{\{z\}}> v_{\{z,w\}}> v_{\{w\}}\).

            \end{itemize}            

       \item {\( v_2 = v_{\{x,w\}} \), \( v_3 = v_{\{y,w\}} \), \( v_4 = v_{\{z,w\}} \), \(v_5 = v_{\{x,y\} } \)}\\
            This case is already considered in \ref{2.3.4}.

        \end{enumerate}

    \end{enumerate}

    \item {\textit{There exist exactly two vertices \( v_1 , v_2\in K \) such that \( d_E(v_1)=d_E(v_2) = 3 \), and there is no vertex in \( K \) such that  \( d_E(v) = 4 \).}}\\

Without loss of generality, let \( v_1 = v_{\{x,y,z\}} \) and \( v_2 = v_{\{y,z,w\}} \). In this case, if we add the vertices \( v_{\{x\}}, v_{\{y\}}, v_{\{z\}} \) or  \( v_{\{y\}}, v_{\{z\}}, v_{\{w\}} \)  or  \( v_{\{x\}}, v_{\{y\}}, v_{\{w\}} \) or  \( v_{\{x\}}, v_{\{z\}}, v_{\{w\}} \) to the clique \( K \) with the existing graph structure, the resulting graph will contain \( T_2 \) as an induced subgraph and for  \( v_{\{x\}}, v_{\{y\}}, v_{\{w\}} \) or  \( v_{\{x\}}, v_{\{z\}}, v_{\{w\}} \) it will contain \( T_6 \). Therefore, we shall exclude these scenarios from consideration in the subsequent subcases while extending the graph.

    \begin{enumerate}

        \item {\textit{{There exists exactly one vertex, say \( v_3 \) in \( K \) such that
 $d_E(v_3) = 2.$}}}
    \begin{enumerate}

            \item {\( v_3 = v_{\{x,y\}}\) \text{or} \(v_{\{y,z\}} \)}\\
        Addition of \( v_{\emptyset} \) does not lead to the formation of an induced subgraph isomorphic to \( T_1 \) or \( T_5 \). Thus, \( v_{\emptyset} \) does not play a role in generating any minimal forbidden subgraphs. Consequently, the maximum possible size of \( K \) in this case is 5, which has already been accounted for in the Theorem \ref{th5}. So, we don't get any new minimal forbidden induced subgraph except the graphs \( T_1, T_2, \ldots, T_8 \) listed above.

            \item {\( v_3 = v_{\{x,w\}} \)}

            \begin{itemize}

    \item Addition of \(v_{\emptyset} \) yields a graph containing \(T_1\) as an induced subgraph.\\
    For the remaining cases,

    \item The maximum possible size of \( K \) is 5, which has already been characterized by Theorem \ref{th5}. So, we don't get any new minimal forbidden induced subgraph except the graphs \( T_1, T_2, \ldots, T_8 \) listed above.

\end{itemize}

        \end{enumerate}

         \item {\textit{There exist exactly two vertices, say \( v_3 \) and \( v_4 \), in \( K \) such that
\[
d_E(v_3) = d_E(v_4) = 2.
\]
}}

\begin{enumerate}

            \item {\( v_3 = v_{\{x,y\}} \), \( v_4 = v_{\{y,z\}} \)}

    \begin{itemize}

            \item Addition of \( v_{\{z\}}, v_{\{w\}} \) yields a graph containing \(T_2\) as an induced subgraph.\\
             For the remaining cases, we provide vertex labelings for five non-isomorphic graphs.

             \item \( E \): $\mathcal{O}$, \( K \): \(v_{\{x\}}>v_{\{x,y\}}> v_{\{x,y,z\}}>v_{\{y,z\}}>  v_{\{y,z,w\}}> v_{\{y\}}> v_{\emptyset}\). 

             \item \( E \): $\mathcal{O}$, \( K \): \(v_{\{y\}}>v_{\{x,y\}}> v_{\{x,y,z\}}>v_{\{y,z\}}>  v_{\{y,z,w\}}> v_{\{z\}}> v_{\emptyset}\).

             \item \( E \): $\mathcal{O}$, \( K \): \(v_{\{x\}}>v_{\{x,y\}}> v_{\{x,y,z\}}>v_{\{y,z\}}>  v_{\{y,z,w\}}> v_{\{z\}}> v_{\emptyset}\).

             \item \( E \): $\mathcal{O}$, \( K \): \(v_{\{x\}}>v_{\{x,y\}}> v_{\{x,y,z\}}>v_{\{y,z\}}>  v_{\{y,z,w\}}> v_{\{w\}}> v_{\emptyset}\).

             \item \( E \): $\mathcal{O}$, \( K \): \(v_{\{y\}}>v_{\{x,y\}}> v_{\{x,y,z\}}>v_{\{y,z\}}>  v_{\{y,z,w\}}> v_{\{w\}}> v_{\emptyset}\). 

    \end{itemize}

            \item {\( v_3 = v_{\{x,y\}} \), \( v_4 = v_{\{z,w\}} \)}

                 \begin{itemize}

            \item The case of adding \( v_{\{x\}}, v_{\{y\}} \) is already considered in \ref{2.1.2}.\\
            
             For the remaining cases, we provide vertex labelings for four non-isomorphic graphs.

             \item \( E \): $\mathcal{O}$, \( K \): \(v_{\{x\}}>v_{\{x,y\}}> v_{\{x,y,z\}}>  v_{\{y,z,w\}}>v_{\{z,w\}}> v_{\{z\}}> v_{\emptyset}\). 

             \item \( E \): $\mathcal{O}$, \( K \): \(v_{\{y\}}>v_{\{x,y\}}> v_{\{x,y,z\}}>  v_{\{y,z,w\}}>v_{\{z,w\}}> v_{\{w\}}> v_{\emptyset}\). 

             \item \( E \): $\mathcal{O}$, \( K \): \(v_{\{x\}}>v_{\{x,y\}}> v_{\{x,y,z\}}>  v_{\{y,z,w\}}>v_{\{z,w\}}> v_{\{w\}}> v_{\emptyset}\).

             \item \( E \): $\mathcal{O}$, \( K \): \(v_{\{y\}}>v_{\{x,y\}}> v_{\{x,y,z\}}>  v_{\{y,z,w\}}>v_{\{z,w\}}> v_{\{z\}}> v_{\emptyset}\).

    \end{itemize}

            \item {\( v_3 = v_{\{x,y\}} \), \( v_4 = v_{\{y,w\}} \)}\label{3.2.3}

    \begin{itemize}

                \item Addition of \( v_{\{z\}}\) yields a graph containing \(T_8\) as an induced subgraph.\\
             For the remaining cases, we provide vertex labelings for three non-isomorphic graphs.

             \item \( E \): $\mathcal{O}$, \( K \): \(v_{\{x\}}>v_{\{x,y\}}> v_{\{x,y,z\}}>  v_{\{y,z,w\}}>v_{\{y,w\}}> v_{\{y\}}> v_{\emptyset}\). 

             \item \( E \): $\mathcal{O}$, \( K \): \(v_{\{y\}}>v_{\{x,y\}}> v_{\{x,y,z\}}>  v_{\{y,z,w\}}>v_{\{y,w\}}> v_{\{w\}}> v_{\emptyset}\). 

             \item \( E \): $\mathcal{O}$, \( K \): \(v_{\{x\}}>v_{\{x,y\}}> v_{\{x,y,z\}}>  v_{\{y,z,w\}}>v_{\{y,w\}}> v_{\{w\}}> v_{\emptyset}\). 

    \end{itemize}

            \item {\( v_3 = v_{\{x,y\}} \), \( v_4 = v_{\{x,z\}} \)}\label{3.2.4}\\
            This graph contains \(T_3\).

            \item {\( v_3 = v_{\{y,z\}} \), \( v_4 = v_{\{x,w\}} \)}\label{3.2.5}

    \begin{itemize}

                \item Addition of \( v_{\{y\}}\) or \( v_{\emptyset}\) yields a graph containing \(T_1\) as an induced subgraph.\\
             For the remaining cases, 
             \item \( E \): $\mathcal{O}$, \( K \): \(v_{\{x,w\}}>v_{\{x\}}>v_{\{x,y,z\}}> v_{\{y,z\}}>  v_{\{y,z,w\}}> v_{\{w\}}\).

     \end{itemize}

            \item {\( v_3 = v_{\{x,w\}} \), \( v_4 = v_{\{x,y\}} \)}\label{3.2.6}

    \begin{itemize}

                \item Addition of \( v_{\{y\}}\) or \( v_{\{z\}}\) or \( v_{\emptyset}\) yields a graph containing \(T_1\) as an induced subgraph.\\
             For the remaining cases, 

             \item \( E \): $\mathcal{O}$, \( K \): \(v_{\{x,w\}}>v_{\{x\}}>v_{\{x,y\}}> v_{\{x,y,z\}}>  v_{\{y,z,w\}}> v_{\{w\}}\). 

     \end{itemize}

 \end{enumerate}

         \item {\textit{There exist exactly three vertices, say \( v_3 \) , \( v_4 \) and \(v_5\) in \( K \) such that
\[
d_E(v_3) = d_E(v_4) = d_E(v_5) = 2.
\]}}

        \begin{enumerate}

             \item {\( v_3 = v_{\{x,y\}} \), \( v_4 = v_{\{y,z\}}\), \( v_5 = v_{\{x,z\}} \)}\\
             This case is considered in \ref{2.3.1}.

             \item {\( v_3 = v_{\{x,y\}} \), \( v_4 = v_{\{y,w\}}\), \( v_5 = v_{\{x,w\}} \)}

             \begin{itemize}

                \item The case of adding \( v_{\{y\}}\) or \(v_{\{z\}}\) or \(v_{\emptyset}\) is considered in \ref{3.2.6}.\\
             For the remaining cases, 
             \item \( E \): $\mathcal{O}$, \( K \): \(v_{\{x,w\}}>v_{\{x\}}>v_{\{x,y\}}> v_{\{x,y,z\}}>  v_{\{y,z,w\}}>v_{\{y,w\}}> v_{\{w\}}\).

     \end{itemize}

            \item {\( v_3 = v_{\{x,y\}} \), \( v_4 = v_{\{x,z\}}\), \( v_5 = v_{\{x,w\}} \)}\\
            This graph contains \(T_3\).

            \item {\( v_3 = v_{\{x,y\}} \), \( v_4 = v_{\{y,z\}}\), \( v_5 = v_{\{y,w\}} \)}

             \begin{itemize}

                \item The case of adding \( v_{\{z\}}\) is already considered in \ref{3.2.3}.\\
             For the remaining cases, we provide vertex labelings for three non-isomorphic graphs.

             \item \( E \): $\mathcal{O}$, \( K \): \(v_{\{x\}}>v_{\{x,y\}}> v_{\{x,y,z\}}>v_{\{y,z\}}>  v_{\{y,z,w\}}>v_{\{y,w\}}> v_{\{y\}}> v_{\emptyset}\). 

            \item \( E \): $\mathcal{O}$, \( K \): \(v_{\{y\}}>v_{\{x,y\}}> v_{\{x,y,z\}}>v_{\{y,z\}}>  v_{\{y,z,w\}}>v_{\{y,w\}}> v_{\{w\}}> v_{\emptyset}\). 

            \item \( E \): $\mathcal{O}$, \( K \): \(v_{\{x\}}>v_{\{x,y\}}> v_{\{x,y,z\}}>v_{\{y,z\}}>  v_{\{y,z,w\}}>v_{\{y,w\}}> v_{\{w\}}> v_{\emptyset}\).

    \end{itemize}

            \item {\( v_3 = v_{\{x,y\}} \), \( v_4 = v_{\{y,z\}}\), \( v_5 = v_{\{x,w\}} \)}

            \begin{itemize}

                \item The case of adding \( v_{\{y\}}\) or \(v_{\{z\}}\) or \(v_{\emptyset}\) is already considered in \ref{3.2.6}.\\
             For the remaining cases, 

             \item \( E \): $\mathcal{O}$, \( K \): \(v_{\{x,w\}}>v_{\{x\}}>v_{\{x,y\}}> v_{\{x,y,z\}}>v_{\{y,z\}}>  v_{\{y,z,w\}}> v_{\{w\}}\).

     \end{itemize}

            \item {\( v_3 = v_{\{x,y\}} \), \( v_4 = v_{\{y,z\}}\), \( v_5 = v_{\{z,w\}} \)}

             \begin{itemize}

            \item The case of adding \( v_{\{x\}},v_{\{y\}} \) is already considered in \ref{2.1.2}.\\
             For the remaining cases, we provide vertex labelings for four non-isomorphic graphs. 

             \item \( E \): $\mathcal{O}$, \( K \): \(v_{\{x\}}>v_{\{x,y\}}> v_{\{x,y,z\}}>v_{\{y,z\}}>  v_{\{y,z,w\}}>v_{\{z,w\}}> v_{\{z\}}>v_{\emptyset}\).

             \item \( E \): $\mathcal{O}$, \( K \): \(v_{\{y\}}>v_{\{x,y\}}> v_{\{x,y,z\}}>v_{\{y,z\}}>  v_{\{y,z,w\}}>v_{\{z,w\}}> v_{\{w\}}>v_{\emptyset}\). 

             \item \( E \): $\mathcal{O}$, \( K \): \(v_{\{x\}}>v_{\{x,y\}}> v_{\{x,y,z\}}>v_{\{y,z\}}>  v_{\{y,z,w\}}>v_{\{z,w\}}> v_{\{w\}}>v_{\emptyset}\). 

             \item \( E \): $\mathcal{O}$, \( K \): \(v_{\{y\}}>v_{\{x,y\}}> v_{\{x,y,z\}}>v_{\{y,z\}}>  v_{\{y,z,w\}}>v_{\{z,w\}}> v_{\{z\}}>v_{\emptyset}\). 

    \end{itemize}

            \item {\( v_3 = v_{\{x,y\}} \), \( v_4 = v_{\{z,w\}}\), \( v_5 = v_{\{x,w\}} \)}

            \begin{itemize}

                \item The case of adding \( v_{\{y\}}\) or \(v_{\{z\}}\) or \(v_{\emptyset}\) is already considered in \ref{3.2.6}.\\
             For the remaining cases, 

             \item \( E \): $\mathcal{O}$, \( K \): \(v_{\{x,w\}}>v_{\{x\}}>v_{\{x,y\}}> v_{\{x,y,z\}}>  v_{\{y,z,w\}}>v_{\{z,w\}}> v_{\{w\}}\).

     \end{itemize}

            \item {\( v_3 = v_{\{x,y\}} \), \( v_4 = v_{\{x,z\}}\), \( v_5 = v_{\{y,w\}} \)}\\
            This case is considered in \ref{3.2.4}.

        \end{enumerate}

        \item {\textit{There exist exactly four vertices, say \( v_2 \), \( v_3 \), \(v_4\) and \(v_5\) in \( K \) such that
\[
d_E(v_2) = d_E(v_3) = d_E(v_4) = d_E(v_5) = 2.
\]}}
         \begin{enumerate}

             \item {\( v_2 = v_{\{x,y\}} \), \( v_3 = v_{\{y,z\}}\), \( v_4 = v_{\{z,w\}}\), \( v_5 = v_{\{x,w\}} \)}

             \begin{itemize}

                \item The case of adding \( v_{\{y\}}\) or \(v_{\{z\}}\) or \(v_{\emptyset}\) is already considered in \ref{3.2.6}.\\
             For the remaining cases, 

             \item \( E \): $\mathcal{O}$, \( K \): \(v_{\{x,w\}}>v_{\{x\}}>v_{\{x,y\}}> v_{\{x,y,z\}}>v_{\{y,z\}}>  v_{\{y,z,w\}}>v_{\{z,w\}}> v_{\{w\}}\).

     \end{itemize}

             \item {\( v_2 = v_{\{x,y\}} \), \( v_3 = v_{\{x,z\}}\), \( v_4 = v_{\{x,w\}}\), \( v_5 = v_{\{y,z\}} \)}\\
             This graph contains \(T_3\).

             \item {\( v_2 = v_{\{x,y\}} \), \( v_3 = v_{\{x,z\}}\), \( v_5 = v_{\{x,w\}}\), \( v_5 = v_{\{y,w\}} \)}\\
             This graph contains \(T_3\).

             \item {\( v_2 = v_{\{x,y\}} \), \( v_3 = v_{\{y,z\}}\), \( v_4 = v_{\{y,w\}}\), \( v_5 = v_{\{x,z\}} \)}\\
             This graph contains \(T_3\).

             \item {\( v_2 = v_{\{x,y\}} \), \( v_3 = v_{\{y,z\}}\), \( v_4 = v_{\{y,w\}}\), \( v_5 = v_{\{x,w\}} \)}

             \begin{itemize}

                \item The case of adding \( v_{\{y\}}\) or \(v_{\{z\}}\) or \(v_{\emptyset}\) is already considered in \ref{3.2.6}.\\
             For the remaining cases, 

             \item \( E \): $\mathcal{O}$, \( K \): \(v_{\{x,w\}}>v_{\{x\}}>v_{\{x,y\}}> v_{\{x,y,z\}}>v_{\{y,z\}}>  v_{\{y,z,w\}}>  v_{\{y,w\}}> v_{\{w\}}\).

     \end{itemize}

        \end{enumerate}

    \end{enumerate}

     \item {\textit{There exist atleast three vertices, say \( v_3 \), \( v_4\), \(v_5\) in \( K \) such that
\[
d_E(v_3) = d_E(v_4)= d_E(v_5)= 3.
\]}}  

 In this case, the graph contains \(T_3\) as an induced subgraph. So, this case is discarded from the further extensions.

    \item {\textit{There exists a vertex \( v_1 \in K \) such that \( d_E(v_1) = 4 \).}}\\

    The vertex \( v_1 \) in this case is a \emph{universal vertex}, i.e., it is adjacent to all other vertices in the graph. By Theorem \ref{co}, a graph \( G \) in this case is word-representable if and only if \( G \setminus v_1 \) is a comparability graph. \cite{golumbic2004algorithmic} provides a characterization of split comparability graphs in terms of forbidden induced subgraphs (as shown in Figure \ref{fig:B123}). 

After adding a universal vertex to \( B_1 \), \( B_2 \), and \( B_3 \), the resulting graphs contain \( T_2 \), \( T_3 \), and \( T_4 \), respectively, as induced subgraphs. Therefore, in this case, \( T_2 \), \( T_3 \), and \( T_4 \) are the only minimal forbidden induced subgraphs.

    \end{enumerate}
\end{proof}
\section{Concluding Remarks}

 In this work, we have provided a characterization for a specific subclass of split graph with \(|E| = 4\). However, the underlying ideas, along with the lists of minimal forbidden induced subgraphs presented here and in~\cite{bonomo2022forbidden}, \cite{chen2022representing}, and \cite{kitaev2021word}, may be helpful for developing characterizations for other fixed sizes of independent sets. Notably, the last two cases remain unchanged regardless of the size of the independent set, so it is sufficient to examine only the first three cases.
 
 To address the question of whether there is a relation between the number of minimal forbidden induced subgraphs for a fixed \(|E|\) and \(|K| = |E| + 1\), we observe the following: for \(|E|=4\) and \(|K|=5\), the numbers are equal, both being 9. However, for \(|E|=3\) and \(|K|=4\), the numbers are unequal, being 3 and 4, respectively. This naturally leads to the further question of whether, in general, the absolute difference between the number of minimal forbidden induced subgraphs corresponding to \(|E|\) and \(|K|=|E|+1\) is at most 1.

	\bibliographystyle{plain}
	\bibliography{ref.bib}
\end{document}